\documentclass{amsart}
\usepackage{xcolor}

\usepackage{amssymb,latexsym,amsmath,amsthm,amsfonts,graphicx}
\numberwithin{equation}{section}
\newtheorem{thm}{Theorem}[section]

\newtheorem{prop}[thm]{Proposition}
\newtheorem{cor}[thm]{Corollary}
\newtheorem{lemma}[thm]{Lemma}

\newtheorem*{remark}{Remark}

\def \remark{\noindent\emph{Remark.}\hspace{ 1.9 pt}}
\newcommand{\RR}{\mathbb{R}}
\newcommand{\CC}{\mathbb{C}}
\newcommand{\ol}{\overline}
\newcommand{\Pm}{P_m}
\newcommand{\Hm}{H_m}

\begin{document}

\title[Dmitry Khavinson, Erik Lundberg, Hermann Render]{Dirichlet's problem with entire data posed on an ellipsoidal cylinder}

\author[Dmitry Khavinson, Erik Lundberg, Hermann Render]{Dmitry Khavinson, Erik Lundberg, Hermann Render}

\maketitle

\section{Introduction}
A function $u$ is said to be harmonic if $\Delta u := \sum_{j=1}^{n}{\frac{\partial^2 u}{\partial x_j^2}} = 0$.
Given a domain $\Omega \subset \RR^n$ with sufficiently smooth boundary $\Gamma:=\partial \Omega$,
and given a function $f$ continuous on $\Gamma$,
the classical Dirichlet problem asks for a function $u$ that is harmonic in $\Omega$
and continuous in $\ol{\Omega}$ such that $u = f$ on $\Gamma$, i.e., $u$ satisfies
\begin{equation} 
\left\{
\vcenter{\openup\jot\ialign
         {\strut\hfil$\displaystyle#$&$\displaystyle{}#$\hfil&\quad#\hfil\cr
\Delta u &= 0 &in\/ $\Omega$,\cr
       u &= f &on $\Gamma$.\cr}}\right.
\label{eq:Dirichlet}
\end{equation}
For a bounded domain $\Omega$ with $\Gamma$ sufficiently smooth, 
it is well known since the early 20th century that a solution $u$ exists and is unique (see for instance \cite[Sec. 2.8]{GilbarTrudinger}).  

Following the pioneering work of H. S. Shapiro \cite{Shapiro89}, 
we suppose that the data function $f$ is the restriction to $\Gamma$ of a ``very nice'' function,
and we inquire whether the solution $u$ can be extended outside its natural domain $\Omega$.
Specifically, consider the following question:

\medskip

{\bf Question.} \emph{If the data function $f$ is entire, is the solution $u$ also entire?}

\medskip

H. S. Shapiro and the first author \cite{KhSh92} showed that the answer is ``yes''
when $\Omega$ is an ellipsoid.
D. Armitage \cite{Armi04} showed further that the order and type of the data are preserved.
The Khavinson-Shapiro conjecture \cite{KhSh92} states that ellipsoids are the only bounded domains for which entire data implies entire solution.
The third author proved this conjecture within a large class of algebraic domains
(namely, those for which the boundary has a defining polynomial with positive leading homogeneous part).
For a detailed discussion on related work, we refer the reader to the survey papers \cite{KhLu2010} and \cite{Render2011}.

In this paper we seek entire solutions for entire data in the case when $\Omega = D \times \RR \subset \RR^n$ is a cylinder with ellipsoidal base $D \subset \RR^{n-1}$.
In other words, starting from the question for ellipsoids resolved positively in \cite{KhSh92},
we let one of the semi-axes tend to infinity.
Under an additional assumption on the order of $f$,
we show that the answer to the above question remains affirmative.

\begin{thm}\label{ordleq1}
Let $\Omega=\{ x\in{\RR}^n : \sum_{j=1}^{n-1} \frac{x_j^2}{a_j^2} < 1\}$, where
$a_j>0$. If $f$~is an entire function on\/ $\CC^n$ with order $\rho(f) < 1$
then the Poisson integral solution of the Dirichlet problem (\ref{eq:Dirichlet}) 
extends to a harmonic function on\/~$\RR^n$ (and an entire function on $\CC^n$).
\end{thm}

The assumption that $\rho(f) < 1$ ensures that the Poisson integral converges.
However, the solution is not unique.
Uniqueness fails for every ellipsoidal base $D$, 
since there are solutions (of order $\rho = 1$) that vanish on $\Gamma$,
such as $\psi_\lambda(x') \exp \left( \sqrt{\lambda} x_n \right)$,
where $\psi_\lambda(x')$ is an eigenfunction (with eigenvalue $\lambda$) 
of the Laplacian in the variables $x' = (x_1,x_2,..,x_{n-1})$
with zero Dirichlet boundary values on $\partial D$.
Yet, as stated in Theorem \ref{thm:yoshida}, under an a priori growth assumption,
the solution is given uniquely by a Poisson integral.

Our proof partly follows the methods used in \cite{KhSh92} while utilizing 
an important additional ingredient---a decay estimate for the Poisson kernel of a cylinder
(see Equation (\ref{eq:PKest}) below).
In the concluding remarks, we state some open questions concerning solutions
that are not represented as Poisson integrals.

\noindent {\bf Acknowledgement.}
This paper resulted from discussions at the conference Dynamical Systems and Complex Analysis VII, May 2015, in Naharia, Israel.
The first two authors gratefully acknowledge NSF support for the conference (grant DMS -- 1464939).

\section{Preliminary results}

Let $S^{n-1}:=\left\{ x\in \mathbb{R}^{n}:\left| x\right|=1\right\}$ be the unit sphere.

Let $\Pm$ denote the space of polynomials in $n$ variables of degree at most $m$,
and $\Hm$ the subspace of homogeneous polynomials of degree $m$.

Let $B_{R}:=\left\{ x\in \mathbb{R}^{n}:\left| x\right| <R\right\} $ be the
open ball in $\mathbb{R}^{n}$ with center $0$ and radius $0<R\leq \infty .$
We say that a series expansion $f = \sum_{m=0}^{\infty}f_{m}\left( x\right)$ in homogeneous polynomials $f_m \in \Hm$
converges compactly on $B_R$ if it converges absolutely and uniformly on each compact subset $K\subset B_{R}$.
If $R$ is infinite then such $f$ extends to an entire function on $\CC^n.$
In the next section, we will use the following basic criterion, see \cite[Prop. 12]{Render}.

\begin{prop}\label{prop:convergence}
Let $f = \sum_{m=0}^{\infty }f_{m}$,
where $f_{m} \in \Hm$.
Then $f$ converges compactly in $B_{R}$ if and only if 
\begin{equation}\label{eqmaxnorm}
\limsup_{m \rightarrow \infty }\left( \max_{\theta \in S%
^{n-1}}\left| f_{m}\left( \theta \right) \right| \right) ^{1/m} \leq R^{-1}.
\end{equation}
Hence, $f$ is entire if
$$\lim_{m \rightarrow \infty } \left( \max_{\theta \in S^{n-1}}\left| f_{m}\left( \theta \right) \right| \right) ^{1/m} = 0.$$
\end{prop}

The \emph{order} of a non-constant entire function $f$ is defined as 
\begin{equation*}
\rho \left( f\right) :=\limsup_{r\rightarrow \infty } \frac{%
\log \log M \left( f ; r\right) }{\log r},
\end{equation*}
where 
\begin{equation*}\label{eq:M1}
M\left( f ; r\right) :=\sup \left\{ \left| f\left( z\right) \right| :z\in \mathbb{C}^{n},\left| z\right| =r\right\} . 
\end{equation*}


If $0<\rho \left( f\right) <\infty $, then the \emph{type} of $f$ is defined by 
\begin{equation*}
\tau \left( f\right) := \limsup_{r\rightarrow \infty } \frac{%
\log M\left( f ; r\right) }{r^{\rho \left( f\right) }}.
\end{equation*}

Let $f$ be an entire function of order $0<\rho <\infty $. 
Then we have the following extension (see \cite[Lemma 5]{Armi04}) to several variables of a classical formula on order and type of an entire function:
\begin{equation}\label{eq:ordtype}
\limsup_{m\rightarrow \infty } \left( m\cdot \max_{|z| =1}\left| f_{m}\left( z \right) \right| ^{\frac{\rho }{m}} \right) =e\rho \tau.
\end{equation}
Thus for given $\varepsilon >0$ we have for all sufficiently large $m \in \mathbb{N}$ 
\begin{equation*}
\max_{|z|=1}\left| f_{m}\left( z \right) \right|
\leq \frac{\left( e\rho \tau +\varepsilon \right) ^{m/\rho }}{m^{m/\rho }}.
\end{equation*}
In particular, if the order of $f(x)$ is less than one, 
there exists $\delta>0$ such that 
$$\max_{|z|=1}\left| f_{m}\left( z \right) \right| \leq \frac{1}{m^{m(1+\delta)}},$$ 
for all large enough $m$.  

In the next section we will need the following lemma from \cite[Lemma 2]{KhSh92}
that can be used to pass estimates from a harmonic polynomial to individual terms in its homogeneous expansion.

\begin{lemma}\label{homog}
Let $v$ be any harmonic polynomial, and $v = v_0 + v_1 + ... + v_m$ its decomposition into homogeneous polynomials.  Then 
$$\mathop {\max }\limits_{\theta \in S^{n-1}}\left| v_{k}\left( \theta \right)\right| \leq C_n k^{n/2}\mathop {\max }\limits_{\theta \in S^{n-1}}\left| v\left( \theta \right)\right| \quad (1 \leq k \leq m),$$
where $C_n$ depends only on $n$.
\end{lemma}

Recall that the solution to the Dirichlet problem (\ref{eq:Dirichlet}) with polynomial data on an ellipsoid is a polynomial.
Using a modification of the celebrated proof of this fact based on linear algebra
(attributed to E. Fischer \cite{Fischer}, see \cite{KhLu2014} for an exposition), 
we show that this result holds also for an ellipsoidal cylinder $\Omega$.

\begin{thm}\label{thm:poly}
Let $\Omega=\{ x\in{\RR}^n : \sum_{j=1}^{n-1} \frac{x_j^2}{a_j^2} < 1\}$ be an ellipsoidal cylinder. 
For each polynomial, $f \in \Pm$, 
there exists a unique polynomial solution $u \in \Pm$ to the Dirichlet problem (\ref{eq:Dirichlet})
with data $f$.
\end{thm}

\begin{proof}
Let $p(x) = \sum_{j=1}^{n-1} \frac{x_j^2}{a_j^2} - 1$ be the defining polynomial
of $\Gamma = \partial \Omega$.
It suffices to show that the Fischer operator, $q\longmapsto F(q):=\Delta (p q)$, is bijective on the space $\Pm$ of polynomials of degree $m$.
Indeed, given data $f$, the solution $u$ can 
then be expressed using inversion of the operator $F$ as
$u = f - p F^{-1}(\Delta f) $.
Since $\Pm$ is finite dimensional, and $F$ is linear, it follows that injectivity and surjectivity are equivalent. 
Suppose the kernel of $F$ is not trivial.  Thus, $\Delta (P q)=0$ for some nonzero polynomial $q$.  
Let $p_2 = \sum_{i=1}^{n-1}{\frac{x_i^2}{a_i^2}}$ and $q_M$ be the leading homogeneous terms of $p$ and $q$ respectively.  
Then $\Delta (p q)=0$ implies $\Delta (p_2 q_M)=0$.  
But this is a contradiction since the Brelot-Choquet theorem \cite{B-C} states that a harmonic polynomial cannot have non-negative factors.
\end{proof}

Let $\lambda$ denote the smallest positive eigenvalue of the Laplacian in the variables $x' = (x_1,x_2,..,x_{n-1})$
and $\psi_\lambda(x')$ the corresponding eigenfunction of the ellipsoid $D$ (with zero Dirichlet boundary values on $\partial D$).
Recall that the cylinder $\Omega$ has a Poisson kernel $K(x,y)$ (which is obtained by
constructing the harmonic Green's function of $\Omega$ and then taking its normal derivative along the boundary).
The following result is a special case of \cite[Thm. 6]{Yoshida} (cf. \cite[Thm. A]{Miya96}).

\begin{thm}[Yoshida \cite{Yoshida}]
\label{thm:yoshida}
Let $f$ be a continuous function on $\Gamma:=\partial \Omega$ satisfying
\begin{equation}\label{eq:growthcond}
\int_{-\infty}^{\infty} \exp \left( -\sqrt{\lambda} |x_n| \right) \left( \int_{\partial D} |f(x',x_n)| d \sigma(x') \right) dx_n < \infty,
\end{equation}
where $d \sigma$ denotes the $(n-2)$-dimensional surface measure on $\partial D$.
Then, letting $dS$ denote the $(n-1)$-dimensional surface measure on $\Gamma$, the Poisson integral
\begin{equation}\label{eq:PK}
U_f(x) = \int_{\Gamma} f(y) K(x,y) dS(y),
\end{equation}
converges at each point $x \in \Omega$ and defines a classical solution to the Dirichlet problem (\ref{eq:Dirichlet}).
Moreover, suppose $u$ is a solution of (\ref{eq:Dirichlet}) that satisfies
\begin{equation}\label{eq:growthcond2}
\lim_{x_n \rightarrow \pm \infty} \exp \left( -\sqrt{\lambda} |x_n| \right) \int_{D}\psi_\lambda(x') u(x',x_n) dx' = 0.
\end{equation}
Then, $u=U_f$ coincides with the Poisson integral.
\end{thm}

\begin{cor}\label{cor:poly}
For an entire data function $f$ with order $\rho < 1$, 
the Dirichlet problem (\ref{eq:Dirichlet}) can be solved by a Poisson integral.
Moreover, if $f \in \Pm$ is a polynomial, then the Poisson integral solution coincides with the polynomial solution $u \in \Pm$ provided by Theorem \ref{thm:poly}.
\end{cor}

\begin{proof}[Proof of Corollary \ref{cor:poly}]
The assumption on $f$ implies that the growth condition (\ref{eq:growthcond}) is satisfied so that there is a Poisson integral solution
to the Dirichlet problem with data $f$.
If $f \in \Pm$ is a polynomial, 
then the polynomial solution $u \in \Pm$ satisfies (\ref{eq:growthcond2}),
and we conclude that $u$ is the Poisson integral solution.
\end{proof}

The growth condition (\ref{eq:growthcond}) is related to 
the exponential decay rate of the Poisson kernel,
an estimate that we state here as it is a key ingredient in the next section.
Let us assume that $S^{n-1} \subset \Omega$.
It follows from \cite[Sec. 7]{Yoshida} that there exists a constant $C > 0$ such that, for all $\theta \in S^{n-1}$ and $y \in \Gamma$, we have:
\begin{equation}\label{eq:PKest}
 K(\theta,y) \leq C\exp \left( -\sqrt{\lambda} |y_n| \right).
\end{equation}
Note that the corresponding inequalities in \cite[Sec. 7, p. 394]{Yoshida} contain additional factors, but since they are continuous functions of $\theta$,
we arrive at the constant $C$ by taking the largest value as $\theta$ varies over the compact set $S^{n-1}$.

\remark
The Radon-Nikodym derivative of the Poisson kernel, known as \emph{harmonic measure},
has a well-known probabilistic interpretation;
namely, the harmonic measure of a subset of the boundary 
determines the probability of Brownian motion first exiting the boundary through that set.
This leads to an intuitive explanation for the fact that the Poisson kernel in a cylinder decays exponentially.
Namely, for $N \in \mathbb{N}$, let $\Omega_N := \{x \in \Omega: |x_n| < N \}$ denote the cylinder $\Omega$ truncated at $x_n = \pm N$.  
The probability that Brownian motion starting from $\theta \in S^{n-1}$ exits $\Omega_N$ 
through one of its caps $|x_n| = N$ 
can be expressed by conditioning on the event of exiting $\Omega_{N-1}$ through one of its caps $|x_n| = N-1$.
This process can be iterated leading to a product of $N-1$ conditional probabilities,
and each of those factors can be estimated uniformly by a constant less than one,
which implies exponential decay.
As this reasoning might suggest, decay estimates for harmonic measure hold in unbounded domains more general than cylinders;
see the recent Ph.D. thesis of K. Ramachandran \cite[Prop. 2.0.8]{Ram2014}.


\section{Proof of Theorem \ref{ordleq1}}
As above, we may assume without loss of generality that $S^{n-1} \subset \Omega$,
since a change of variables does not effect the magnitude of the order of $f$.
Let us write $f(x) = \sum_{k=0}^{\infty}{f_m(x)}$ as a series of homogeneous polynomials $f_m$ of degree $m$.  
Let $M_m = \max_{|z|=r}\left| f_{m}\left( z \right) \right|$.  
Since the order of $f(x)$ is less than one, it follows from the comments after equation (\ref{eq:ordtype})
that there exists $\delta>0$ such that $M_m \leq \frac{1}{m^{m(1+\delta)}}$ for all large enough $m$.
Let $u_m \in \Pm$ be the solution, guaranteed by Theorem \ref{thm:poly}, to the Dirichlet problem with data $f_m$, and $u_m = u_{m,0}+u_{m,1}+...+u_{m,m}$ its decomposition into homogeneous polynomials. 
Then by Corollary \ref{cor:poly} $u_m$ has the Poisson integral representation 
$u_m(x) = \int_{\Gamma}{f_m(y)\cdot K(x,y)d\sigma_y}$.  
We will show that the solution $u(x) = \sum_{m=0}^{\infty}{\sum_{k=0}^{m}{u_{m,k}(x)}}$ is entire.  

The degree $j$ homogeneous term of $u$ is given by $\sum_{m=j}^{\infty}{u_{m,j}}$.
Assume that each of these sums converges for $\theta \in S^{n-1}$
(this is justified throughout the estimates that follow).  
Then, according to Proposition \ref{prop:convergence}, $u$ is entire if 
\begin{equation}\label{eq:obj}
\limsup_{j\rightarrow \infty }\left( \max_{\theta \in S^{n-1}} \left| \sum_{m=j}^{\infty}{u_{m,j}(\theta)} \right| \right) ^{1/j} = 0
\end{equation}
Let us first estimate $|u_m(\theta)|$ for $\theta \in S^{n-1}$:

\begin{align*}
|u_m(\theta)| =  \left| \int_{\Gamma}{f_m(y)\cdot K(\theta,y)dy} \right| &\leq  \int_{\Gamma}{\left|f_m(y)\right|\cdot K(\theta,y)dy} \\
\text{(since } f_m \text{ is homogeneous)} \quad &= \int_{\Gamma}{||y||^m \left| f_m \left(\frac{y}{||y||}\right) \right|\cdot K(\theta,y)dy} \\
&\leq \int_{\Gamma}{||y||^m M_m \cdot K(\theta,y)dy} \\
\text{(by the growth estimate (\ref{eq:PKest}))} \quad &\leq \int_{\Gamma}{||y||^m M_m \cdot C \exp(-\sqrt{\lambda}|y_n|) dy} \\
&\leq C M_m 2^m \int_{\Gamma}{(A^m+|y_n|^m) \cdot \exp(-\sqrt{\lambda}|y_n|) dy},
\end{align*}
where $A$ denotes the maximum semi-axis of the base of the cylinder.  Let $C_e$ be the $(n-2)$-dimensional 
surface area of the ellipsoidal base of the cylinder.  
Then this last integral can be calculated exactly:
$$ 2 C C_e M_m 2^m \int_{0}^{\infty}{(A^m + y_n^m) \exp(-\sqrt{\lambda} \cdot y_n)dy_n}= \frac{2 C C_e}{\sqrt{\lambda}} M_m 2^m (A^m + m!/\lambda^{m/2}).$$

Write $\hat{C}=\frac{4 C C_e}{\sqrt{\lambda}}$ and $B= 2 \cdot \max \{ A,1/\sqrt{\lambda} \}$.  Then we have
$$ \max_{\theta \in S^{n-1}} \left|u_m(\theta)\right| \leq \hat{C} M_m B^m m! .$$

Applying Lemma \ref{homog} while defining $C':= \hat{C} \cdot C_n$, we have:
\begin{equation}\label{eq:basic}
\max_{\theta \in S^{n-1}} \left|u_{m,j}(\theta)\right| \leq  C_n j^{n/2} \max_{\theta \in S^{n-1}} \left|u_m(\theta)\right| = C' j^{n/2} M_m B^m m! .
\end{equation}
Next we show $u$ is entire by establishing the condition (\ref{eq:obj}).
Starting with a simple estimate followed by term-wise application of the above, we obtain:
\begin{align*}
\max_{\theta \in S^{n-1}} \left| \sum_{m=j}^{\infty}{u_{m,j}(\theta)} \right| &\leq  \sum_{m=j}^{\infty} \max_{\theta \in S^{n-1}}  \left|{u_{m,j}(\theta)} \right| \\
\text{(by estimate (\ref{eq:basic}))}  \quad  &\leq C' j^{n/2} \sum_{m=j}^{\infty}{ M_m B^m m!} \\
&\leq C'' j^{n/2} \sum_{m=j}^{\infty}{ \frac{m^{m+1}}{m^{m(1+\delta)}} (B/e)^m},
\end{align*}
where in the last line we have used the elementary estimate, $m! \leq  e m^{m+1} e^{-m}$,
along with the estimate $M_m = O(\frac{1}{m^{m(1+\delta)}})$.
Let us split $\delta = \delta_1 + \delta_2$ so that $\delta_1, \delta_2 >0$ in order to rewrite the last sum above as:
$$ C'' j^{n/2} \sum_{m=j}^{\infty}{ \frac{m^{m}}{m^{m(1+\delta_1)}} \frac{m \cdot (B/e)^m}{m^{m \delta_2}}}. $$

We can bound $\frac{m \cdot (B/e)^m}{m^{m \delta_2}}$ by a constant.  Thus, this sum is at most
$$ C''' j^{n/2} \sum_{m=j}^{\infty}{ \left(\frac{1}{m^{\delta_1}}\right)^m } \leq C''' j^{n/2} \sum_{m=j}^{\infty}{ \left(\frac{1}{j^{\delta_1}}\right)^m } = C''' j^{n/2} \left(\frac{1}{j^{\delta_1}}\right)^j \frac{1}{1-\frac{1}{j^{\delta_1}} }.$$
Taking the $jth$ root, we have $ \left( C''' j^{n/2} \frac{1}{1-\frac{1}{j^{\delta_1}} } \right)^{1/j}  \frac{1}{j^{\delta_1}}$ which converges to zero as $j \rightarrow \infty$.
Thus, condition (\ref{eq:obj}) is satisfied.  This concludes the proof of Theorem \ref{ordleq1}.

\section{Concluding remarks}

We have shown that, with entire data of order $\rho(f) < 1$, 
the Poisson integral solution extends as an entire function.
Let us reiterate (see the above remark after Theorem \ref{ordleq1}) that this solution is not unique;
in fact, the space of harmonic functions vanishing on the boundary is infinite dimensional.
This leads to the question of whether every solution is entire.
By subtracting the solution that has already been shown to be entire,
this reduces the problem to considering zero boundary data.
Thus, if a function harmonic in $\Omega$ vanishes on the boundary $\Gamma$,
does it extend as an entire harmonic function?
This question has been asked by the first author at conferences.
For $n=2$ (the ``cylinder'' becomes a strip) the answer is ``yes'' as follows from an argument based on the reflection principle.
The question was resolved positively for circular cylinders by S. Gardiner and the third author \cite{GardRend}. 
The case of elliptic cylinders remains open; we conjecture the same outcome in that case.

It seems likely that the same method in the above proof of Theorem \ref{ordleq1}
can be applied also to entire data $f$ 
satisfying a restriction on growth only in the $x_n$-direction such as having, for every $\varepsilon>0$ and $x'=(x_1,x_2,..,x_{n-1})$,
a constant $c(x')$ such that
$$ |f(x)| \leq c(x') \cdot \exp( {\sqrt{\lambda} |x_n|} ) .$$

Miyamoto proved in \cite[Thm. 2]{Miya96} that for any continuous function 
on a cylinder the Dirichlet problem has a (non-unique) classical solution.
It is then natural to ask: for arbitrary entire data, is every solution entire?
The $n=2$ case of a strip again provides some evidence in this direction,
as well as spherical cylinders in $\RR^4$, using the techniques from \cite{Khav91} to reduce the problem to the planar case.
This leads us to conjecture an affirmative answer to this question as well.

In a forthcoming paper, we will address these questions for domains in $\RR^n$ bounded by a pair of parallel hyperplanes.

\bibliographystyle{amsplain}

\end{document}